\documentclass{article}
\usepackage{graphicx}
\usepackage{amsmath}
\RequirePackage[OT1]{fontenc}
\RequirePackage{amsthm,amsmath,amssymb}
\RequirePackage[colorlinks]{hyperref}

\theoremstyle{plain}
\newtheorem{thm}{Theorem}[section]
\newtheorem{lem}[thm]{Lemma}
\newtheorem{cor}[thm]{Corollary}
\newtheorem{conj}[thm]{Conjecture}

\newtheorem{ex}[thm]{Example}
\newtheorem{rem}[thm]{Remark}

\newcommand{\ve}[0]{\varepsilon}

\newcommand{\bbZ}{{\rm \rlap Z\kern 2.2pt Z}}

\newcommand{\Cov}{{\rm Cov}}
\newcommand{\Var}{\mathrm {Var}}

\newcommand{\Exp}{{\rm E}}

\newcommand{\Ex}{{\rm E}}

\newcommand{\convdistr}{\stackrel{\scriptstyle  d}{\to}}

\newcommand{\convweak}{\Rightarrow }

\newcommand{\convas}{\stackrel{\scriptstyle a.s.}{\to}}

\newcommand\T{\rule{0pt}{2.6ex}}
\newcommand\B{\rule[-1.2ex]{0pt}{0pt}}

\begin{document}

\title{Empirical process of residuals for regression models with long memory errors\protect}

\author{ Pawe{\l} Lorek \and Rafa{\l} Kulik }

\maketitle

\begin{abstract}
We consider the residual empirical process in random design
regression with long memory errors. We establish its limiting
behaviour, showing that its rates of convergence are different from
the rates of convergence for to the empirical process based on
(unobserved) errors. Also, we study a residual empirical process
with estimated parameters. Its asymptotic distribution can be used
to construct Kolmogorov-Smirnov, Cram\'{e}r-Smirnov-von Mises, or
other goodness-of-fit tests. Theoretical results are justified by
simulation studies.
\end{abstract}

\section{Introduction}
Consider a random design regression model,
\begin{equation}\label{eq-regression-randomdesign}
Y_i=m(X_i)+\varepsilon_i, \qquad i=1,\ldots,n ,
\end{equation}
where $\{X,X_i,i\ge 1\}$ is a stationary sequence of random
variables with a density $f=f_X$, independent of a centered,
stationary long memory error sequence
$\{\varepsilon,\varepsilon_i,-\infty<i<\infty\}$, with a
distribution $F_{\varepsilon}$ and density $f_{\varepsilon}$. The
goal of this paper is to study the asymptotic properties of the
empirical process of residuals,
$$
\hat K_n(x):=\sum_{i=1}^n \left(\mathbf{1}_{\{\hat\varepsilon_i\le
x\}}-F_{\varepsilon}(x)\right),
$$
where $\hat\varepsilon_i=Y_i-\hat m(X_i)$ and $\hat m(\cdot)$ is an estimator of the function $m(\cdot)$.\\

Residual-based inference is a standard tool in regression analysis.
With this in mind, several authors considered empirical process of
residuals in case of independent random variables or weakly dependent stationary time series,
see e.g. \cite{Boldin1982}, 
\cite{KoulOssiander1994}, \cite{Bai1994}, \cite{Schick2004}, \cite{Cheng2005}, \cite{Schick2009}, just to mention few.\\

As for regression models with long memory errors, in
\cite{ChanLing2008}, the authors obtained that in case of a
parametric regression, $m(x)=\beta_0+\beta_1 x$, with \textit{a
known intercept}, the limiting behaviour of $\hat K_n(\cdot)$ is
similar to the limiting behaviour of
$$
K_n(x):=\sum_{i=1}^n \left(\mathbf{1}_{\{\varepsilon_i\le
x\}}-F_{\varepsilon}(x)\right),
$$
in the sense that $\sigma_{n,1}^{-1}K_n(\cdot)$ and
$\sigma_{n,1}^{-1}\hat K_n(\cdot)$ converge weakly to, respectively,
$f_{\varepsilon}(x)Z_1$, where $Z_1$ is standard normal and
$\sigma_{n,1}$ is an appropriate scaling factor. However, if one
considers a parametric regression when both slope and intercept are
unknown, from the latter paper one can only conclude that
$$\sigma_{n,1}^{-1}\sup_{x\in\mathbb{R}}|\hat K_n(x)|=o_P(1).$$

To see intuitively why this is the case, consider for a moment a simple model 
$Y_i=\mu+\varepsilon_i$. Estimate $\mu$ by the sample mean $\bar Y$ and compute residuals $\hat\varepsilon_i=Y_i-\bar Y$. Then 
$$
\hat K_n(x)=K_n(x+\bar Y)+\left(F_{\varepsilon}(x+\bar Y)-F_{\varepsilon}(x)\right).
$$
From a general theory for empirical processes based on long memory random variables $\varepsilon_i$, $i\ge 1$, we conclude 
$$
\sigma_{n,1}^{-1}\sup_x\left|K_n(x)+f_{\varepsilon}(x)\sum_{i=1}^n\varepsilon_i\right|=o_P(1) ,
$$
see e.g. \cite{DehlingTaqqu}, \cite{HoHsing1996}, \cite{Wu2003}. On the other hand, using the Taylor's expansion, 
$$
\left(F_{\varepsilon}(x+\bar Y)-F_{\varepsilon}(x)\right)\approx f_{\varepsilon}(x) \sum_{i=1}^n\varepsilon_i .
$$
Therefore, the contribution of $\bar \varepsilon=\sum_{i=1}^n\varepsilon_i$ cancels out and asymptotic behaviour of $\hat K_n(x)$ cannot be concluded from that of $K_n(x)$. See \cite{Kulik2008b} and \cite{KoulSurgailis2010} for precise results along these lines. \\

The main goal of this paper is to establish a general theory on asymptotic behaviour
for $\hat K_n(\cdot)$. In particular, this theory is applied to the parametric regression and a nonparametric regression; the latter in a longer, arxiv version of the paper. We
will show in this paper, that convergence properties of $\hat
K_n(\cdot)$ may be completely different from the asymptotics of
$K_n(\cdot)$. To do this, we will establish a second order expansion
for $\hat K_n(\cdot)$ (see Theorems \ref{thm:main} and
\ref{thm:main-weak}).

The established results can be used, in principle, to test whether
the errors $\varepsilon_1,\ldots,\varepsilon_n$ are consistent with
a given distribution $F_\varepsilon{}$. If $F_{\varepsilon}$ belongs
to a one-parameter family
$\{F_{\varepsilon}(\cdot,{\theta}),\theta\in \mathbb{R} \}$, then
one needs to know the value of the parameter $\theta$. Therefore, we
discuss asymptotic properties of an empirical process of residuals,
when a parameter $\theta$ is estimated.
The appropriate limit theorems are established in Section \ref{sec:estimated}. Our theoretical results are confirmed by small simulation studies in Section \ref{sec:simulation}.\\

The results for empirical processes in Sections \ref{sec:parametric}
can be applied directly to establish
limiting behaviour of quantiles (see \cite[Section 5]{HoHsing1996}). Furthermore, in a spirit of
\cite[Section 3]{HoHsing1996}, our results should be applicable to the error density estimation. 
However, a precise proof requires at least third order expansion of the residual-based empirical process (see Section \ref{sec:density}). 
Finally, it would be interesting to establish
corresponding results in case of fixed-design regression.
\section{Preliminaries: LRD error sequence}\label{sec:errors-properties}
In the sequel, $F_U(\cdot)$, $f_U(\cdot)$ denote a distribution and
a density, respectively, of a given random variable $U$. Also, if
$U$ has finite mean, we denote $U^* =U-\Ex[U]$. \\

We shall consider the following assumption on the error sequence:
\begin{itemize}
\item[{\rm (E)}] $\varepsilon_i$, $i\ge 1$, is an infinite order moving average
$$
\varepsilon_i=\sum_{k=0}^{\infty}c_k\eta_{i-k},\qquad \mbox{\rm with
} c_0=1,
$$
where $\eta_i$, $-\infty<i<\infty$, is a sequence of centered
i.i.d.~random variables, independent of $X_i$, $i\ge 1$. We assume
that $\Exp[\varepsilon^4]<\infty$, $\Exp[\varepsilon_1^2]=1$, and
for some $\alpha_{\varepsilon}\in (0,1)$, $c_k\sim
k^{-(\alpha_{\varepsilon}+1)/2}L_0(k)$ as $k\to\infty$, where
$L_0(\cdot)$ is slowly varying at infinity.
\end{itemize}
Let
\begin{equation}\label{eq:enr}
\varepsilon_{n,r}=\sum_{i=1}^n\sum_{1\le
j_1<\cdots<j_r=1}^n\prod_{s=1}^rc_{j_s}\eta_{i-j_s}.
\end{equation}
In particular, $\varepsilon_{n,1}=\sum_{i=1}^n\varepsilon_{i}$ and
if $r\alpha_{\varepsilon}<1$,
\begin{equation}\label{eq:var-enr}
\sigma_{n,r}^2:=\Var (\varepsilon_{n,r})\sim
n^{2-r\alpha_{\varepsilon}}L_0^{2r}(n).
\end{equation}
From \cite{HoHsing1996} we know that for
$r<\alpha_{\varepsilon}^{-1}$, as $n\to\infty$,
\begin{equation}\label{eq:enr-distr}
\sigma_{n,r}^{-1}\varepsilon_{n,r}\convdistr Z_r,\qquad r=1,2,
\end{equation}
where $Z_r$ is a random variable which can be represented by
a multiple Wiener-It\^{o} integral. In particular, $Z_1$
is standard normal. Moreover, the random variables $Z_1,\ldots,Z_p$
are uncorrelated, see e.g. \cite[Eq. (1.22)]{KoulSurgailis1997}. We also note that the convergence in (\ref{eq:enr-distr}) holds jointly. 

Furthermore, let
$$
S_{n,p}(x)=\sum_{i=1}^n\left(\mathbf{1}_{\{\varepsilon_i\le
x\}}-F_{\varepsilon}(x)\right)+\sum_{r=1}^p(-1)^{r-1}F_{\varepsilon}^{(r)}(x)\varepsilon_{n,r}.
$$
Assume that $F_{\eta}(\cdot)$ is 5 times differentiable with
bounded, continuous and integrable derivatives. We note in passing
that these properties are transferable to $F_{\varepsilon}(\cdot)$.
Following \cite[Theorem 3]{Wu2003} and \cite[Theorem
2.2]{HoHsing1996} we conclude, in particular, that for
$\alpha_{\varepsilon}<1/2$,
\begin{equation}\label{eq:Wu}
\sigma_{n,2}^{-1}S_{n,1}(x)\convweak f_{\varepsilon}^{(1)}(x)Z_2,
\end{equation}
where $Z_2$ is the same random variable as in (\ref{eq:enr-distr}).
Otherwise, if $\alpha_{\varepsilon}>1/2$, then
\begin{equation}\label{eq:Wu-weak}
n^{-1/2}S_{n,1}(x)\convweak W_1(x),
\end{equation}
where $\{W_1(x),x\in\mathbb{R}\}$ is a Gaussian process and $\Rightarrow$ denotes weak convergence in $D[0,1]$.
Furthermore, for $\alpha_{\varepsilon}>1/3$,
\begin{equation}\label{eq:Wu-3}
\sigma_{n,2}^{-1}\sup_{x\in\mathbb{R}}|S_{n,2}(x)|\convas 0.
\end{equation}
The structure of this Gaussian process and
its covariance is given in a rather complicated form; see
\cite{Wu2003} for more details.
\section{Results}
Let
$$
\Delta_i:=\varepsilon_i-\hat\varepsilon_i=\varepsilon_i-(Y_i-\hat
m(X_i))=\hat m(X_i)-m(X_i), \qquad {\bf
\Delta}=(\Delta_1,\ldots,\Delta_n).
$$
\subsection{Empirical process of residuals: $\alpha_{\varepsilon}<1/2$}
The following result provides an uniform expansion of the process
$\hat K_n(\cdot)$ and forms a basis for further analysis.
\begin{thm}\label{thm:main}
Assume {\rm (E)} with $\alpha_{\varepsilon}<1/2$. Assume that
$F_{\eta}(\cdot)$ is 3 times differentiable with bounded, continuous
and integrable derivatives. Suppose that ${\bf\Delta}$ can be
written as $\Delta_0{\bf 1}+(\Delta_{01},\ldots,\Delta_{0n})$, where
\begin{equation}\label{eq:cond-1a}
\delta_n+\frac{n^2\delta_n}{\sigma_{n,2}^4}+\frac{n^2\delta_n^2}{\sigma_{n,2}^3}+\frac{n\delta_n}{\sigma_{n,2}^3}\to
0;
\end{equation}
\begin{itemize}
\item $\frac{1}{\delta_n}\Delta_0\stackrel{{\rm d}}{\to} V$, where $V$ is a nondegenerate random variable;
\item $|\Delta_{0i}|=o_P(\delta_n)$, uniformly in $i$.
\end{itemize}
Then
\begin{eqnarray*}
\lefteqn{\sup_{x\in\mathbb{R}}\left|
\hat K_n(x)-K_n(x)-f_{\varepsilon}(x)\sum_{i=1}^n\Delta_i-\frac{1}{2}f_{\varepsilon}^{(1)}(x)\sum_{i=1}^n\Delta_i^2+f_{\varepsilon}^{(1)}(x)\Delta_0 \varepsilon_{n,1}\right|}\\
&=&O_P(\delta_n^{1-\nu}\sigma_{n,2})+o_P\left(\delta_n\sigma_{n,1}\right)+O_P\left(\sum_{i=1}^n\Delta_i^3\right).\qquad\qquad\qquad\qquad\qquad\qquad
\end{eqnarray*}
\end{thm}
In principle, this result is very similar to \cite[Theorem 2.1]{ChanLing2008}. However, we provide $o_P(\cdot)$ rates of the approximation. This is crucial to establish limit theorems for the process $\hat K_n(\cdot)$.\\

To have some intuition, let us write
\begin{eqnarray}\label{eq:basic-decomposition}
\lefteqn{\hat K_n(x)-K_n(x)
= \sum_{i=1}^n\mathbf{1}_{\{\varepsilon_i\le x+\Delta_i\}}-\sum_{i=1}^n\mathbf{1}_{\{\varepsilon_i\le x\}} }\nonumber\\
&=&  \sum_{i=1}^n\left(\mathbf{1}_{\{\varepsilon_i\le
x+\Delta_i\}}-F_{\varepsilon}(x+\Delta_i)\right) -
\sum_{i=1}^n\left(\mathbf{1}_{\{\varepsilon_i\le
x\}}-F_{\varepsilon}(x)\right)\nonumber\\
&&\qquad\qquad+\sum_{i=1}^n\left(F_{\varepsilon}(x+\Delta_i)-F_{\varepsilon}(x)\right).
\end{eqnarray}
From Theorem \ref{thm:main} and (\ref{eq:Wu-3}) we conclude for
$\alpha_{\varepsilon}<1/2$ that, uniformly in $x$,
\begin{eqnarray}\label{eq:conclusion}
\lefteqn{\hat K_n(x)=K_n(x)+f_{\varepsilon}(x)\sum_{i=1}^n\Delta_i+\frac{1}{2}f_{\varepsilon}'(x)\sum_{i=1}^n\Delta_i^2-f_{\varepsilon}^{(1)}(x)\Delta_0 \varepsilon_{n,1}}\nonumber\\
&&+o_P(\sigma_{n,2}+\delta_n\sigma_{n,1})+O_P\left(\sum_{i=1}^n\Delta_i^3\right)\nonumber\\
&=& -f_{\varepsilon}(x)\sum_{i=1}^n\varepsilon_i+f_{\varepsilon}(x)\sum_{i=1}^n\Delta_i+f_{\varepsilon}^{(1)}(x)\varepsilon_{n,2}+\frac{1}{2}f_{\varepsilon}^{(1)}(x)\sum_{i=1}^n\Delta_i^2-f_{\varepsilon}^{(1)}(x)\Delta_0 \varepsilon_{n,1}\nonumber\\
&&\qquad\qquad+o_P(\sigma_{n,2}+\delta_n\sigma_{n,1})+O_P\left(\sum_{i=1}^n\Delta_i^3\right).
\end{eqnarray}
We note in passing that in order to obtain the above expansion via
(\ref{eq:Wu-3}) one has to assume that $F_{\eta}(\cdot)$ is 5 times
differentiable. 
\\

As we will see below (Section \ref{sec:parametric}), it may happen that the first order
contribution
$$-f_{\varepsilon}(x)\sum_{i=1}^n\varepsilon_i+f_{\varepsilon}(x)\sum_{i=1}^n\Delta_i$$
is negligible. In other words, the rates of convergence of $\hat
K_n(\cdot)$ \textit{will be different} from those for $K_n(\cdot)$.
The rates of convergence will be determined by the second order term
$$
f_{\varepsilon}^{(1)}(x)\varepsilon_{n,2}+\frac{1}{2}f_{\varepsilon}^{(1)}(x)\sum_{i=1}^n\Delta_i^2-f_{\varepsilon}^{(1)}(x)\Delta_0
\varepsilon_{n,1}.
$$
\subsection{Empirical process of residuals: $\alpha_{\varepsilon}>1/2$}
Let $\xi_i=\varepsilon_i-\eta_i$. Define $\xi_{n,r}$ in the
analogous way as $\varepsilon_{n,r}$; see (\ref{eq:enr}).

\begin{thm}\label{thm:main-weak}
Assume {\rm (E)} with $\alpha_{\varepsilon}>1/2$. Assume that
$F_{\eta}(\cdot)$ is 3 times differentiable with bounded, continuous
and integrable derivatives. Suppose that ${\bf\Delta}$ can be
written as $\Delta_0{\bf 1}+(\Delta_{01},\ldots,\Delta_{0n})$, where
\begin{equation}\label{eq:cond-1b}
\delta_n+\sqrt{n}\delta_n^2\to 0;
\end{equation}
\begin{itemize}
\item $\frac{1}{\delta_n}\Delta_0\stackrel{{\rm d}}{\to} V$, where $V$ is a nondegenerate random variable;
\item $|\Delta_{0i}|=o_P(\delta_n)$, uniformly in $i$.
\end{itemize}
Then
\begin{eqnarray*}
\hat K_n(x)& = & K_n(x)+f_{\varepsilon}(x)\sum_{i=1}^n\varepsilon_i +\left(f_{\varepsilon}(x)\sum_{i=1}^n\Delta_i-f_{\varepsilon}(x)\sum_{i=1}^n\varepsilon_i-f^{(1)}_{\varepsilon}(x)\Delta_{0}\xi_{n,1}\right)   \nonumber\\
&& +O\left(\sum_{i=1}^n\Delta_i^2\right)+o_P(
\sqrt{n})+O_P\left(\delta_n\sigma_{n,1}\right)
\end{eqnarray*}
where
$n^{-1/2}\left(K_n(x)+f_{\varepsilon}(x)\sum_{i=1}^n\varepsilon_i\right)$
converges weakly to $W_1(x)$ from {\rm (\ref{eq:Wu-weak})}.
\end{thm}
\subsection{Application to parametric regression}\label{sec:parametric}
The results of Theorems \ref{thm:main} and \ref{thm:main-weak} are
the tools to establish a limit theorem for $\hat K_n(\cdot)$ in case
of parametric model
\begin{equation}\label{eq:parametric}
m(x)=\beta_0+\beta_1 x .
\end{equation}
We assume that the regression parameters are estimated using
standard least squares. We make the following assumption on the
predictors $X_i$, $i\ge 1$:
\begin{itemize}
\item[{\rm (P)}] $X_i$, $i\ge 1$, is a random sequence such that
$\sup_i \Ex[|X_i|+|\bar X|]<\infty$.
\end{itemize}

\begin{cor}\label{cor:par}
Assume {\rm (P)} and {\rm (E)} and that
\begin{equation}\label{eq:hatbeta1}
\hat\beta_1-\beta_1=o_P(\sigma_{n,1}/n).
\end{equation}
Assume that $F_{\eta}(\cdot)$ is 5 times differentiable with
bounded, continuous and integrable derivatives.
\begin{itemize}
\item[{\rm (a)}]
If $\alpha_{\varepsilon}<1/2$, then
$$
\frac{1}{\sigma_{n,2}}\hat K_n(x)\convweak f_{\varepsilon}^{(1)}(x)
(Z_2-\frac{1}{2}Z_1^2),
$$
where $Z_1,Z_2$ are defined in {\rm (\ref{eq:enr-distr})}.
\item[{\rm (b)}]
If $\alpha_{\varepsilon}>1/2$, then $ n^{-1/2}\hat K_n(x)\convweak
W_1(x) \; . $
\end{itemize}
\end{cor}
\begin{rem}\label{rem:changed-rates}{\rm
Note that the rate of convergence $\sigma_{n,1}$ for the original
process $K_n(\cdot)$ changes to $\sigma_{n,2}$ or $\sqrt{n}$ for
$\hat K_n(\cdot)$. The similar phenomena was observed in a context
of empirical processes with estimated parameters in
\cite{Kulik2008b} (see also \cite{BeranGhosh1991}). Note further
that a possible LRD of predictors does not play any role.

Furthermore, from the proof of Corollary \ref{cor:par} below, we may
conclude that in case $\beta_0=0$ the limiting behaviour of $K_n(x)$
and $\hat K_n(x)$ is the same. In other words, for the model
(\ref{eq:parametric}) with $\beta_0=0$, we have (see also
\cite{ChanLing2008})
$$
\sigma_{n,1}^{-1}\hat K_n(x)\convweak f_{\varepsilon}(x)Z_1.
$$
}
\end{rem}
\begin{rem}{\rm
The condition (\ref{eq:hatbeta1}) can be verified for many
stationary sequences. In particular, if $X_i$, $i\ge 1$, is LRD
linear sequence with parameter $\alpha_X$, then the rate of
convergence of $(\hat\beta_1-\beta_1)$ is either $\sqrt{n}$ or
$n^{(\alpha_X+\alpha_{\varepsilon})/2}$, for
$\alpha_X+\alpha_{\varepsilon}>1$ or
$\max(\alpha_X,\alpha_{\varepsilon})<1/2$, respectively; see
\cite{HidalgoRobinson1997} and \cite{GuoKoul2008}. }
\end{rem}
\textit{Proof of Corollary \ref{cor:par}.} Least squares estimation
leads to the following expressions:
\begin{eqnarray}\label{eq:beta-estimator}
\qquad\qquad\hat\beta_1-\beta_1=\frac{1}{s_n}\left(\frac{1}{n}\sum_{j=1}^nX_j\varepsilon_j-\bar
X\bar{\varepsilon}\right),\quad\hat\beta_0-\beta_0=\bar{\varepsilon}-\bar
X(\hat\beta_1-\beta_1),
\end{eqnarray}
where $\bar X$ and $\bar{\varepsilon}$ are sample means based on
$X_1,\ldots,X_n$ and $\varepsilon_1,\ldots,\varepsilon_n$,
respectively, and $s_n=\frac{1}{n}\sum_{j=1}^n(X_j-\bar X)^2$. We
have
\begin{eqnarray}\label{eq:1}
\Delta_i&=&\hat m(X_i)-m(X_i)
=(\hat\beta_0-\beta_0)+(\hat\beta_1-\beta_1)X_i=\bar{\varepsilon}+(\hat\beta_1-\beta_1)(X_i-\bar
X).\nonumber\\
\end{eqnarray}
From (\ref{eq:var-enr}) we conclude that
\begin{equation}\label{eq:rate}
\bar\varepsilon=O_P(\sigma_{n,1}/n), \qquad\sigma_{n,1}^2/n\sim
\sigma_{n,2}, \quad \mbox{\rm as } n\to\infty .
\end{equation}
From (\ref{eq:hatbeta1}) and Assumption (P) we conclude $
\Delta_i=\bar{\varepsilon}+o_P(\sigma_{n,1}/n)O_P(1)$. Let now
$\delta_n= \sigma_{n,1}/n$. It is straightforward to check that such
$\delta_n$ fulfills (\ref{eq:cond-1a}). Therefore, the conditions of
Theorem \ref{thm:main} are fulfilled with $\Delta_0=\bar\varepsilon$
and
$V=Z_1$.

Furthermore, from (\ref{eq:1}),
$\sum_{i=1}^n\Delta_i=n\bar\varepsilon=\varepsilon_{n,1}=\sum_{i=1}^n\varepsilon_i
$ and via (\ref{eq:rate}),
\begin{equation}\label{eq:delta2}
\sum_{i=1}^n\Delta_i^2=n\bar\varepsilon^2+n\bar\varepsilon
o_P(\sigma_{n,1}/n)+o_P(n\sigma_{n,1}^2/n^2)
=n\bar\varepsilon^2+o_P(\sigma_{n,2}).
\end{equation}
Consequently, noting that $\delta_n\sigma_{n,1}\sim \sigma_{n,2}$
and $n\bar\varepsilon^2=\bar\varepsilon \varepsilon_{n,1}$, the
expansion (\ref{eq:conclusion}) reads
\begin{eqnarray}\label{eq:conclusion-1a}
\hat K_n(x) &=&
f_{\varepsilon}^{(1)}(x)\varepsilon_{n,2}-\frac{1}{2}f_{\varepsilon}^{(1)}(x)n\bar\varepsilon^2
 +o_P(\sigma_{n,2})=:S_n(x)+o_P(\sigma_{n,2}),
\end{eqnarray}
uniformly in $x$. 
The result of part (a) follows now from (\ref{eq:enr-distr}).
\\

As for part (b), we recall that
$\sum_{i=1}^n\Delta_i-\sum_{i=1}^n\varepsilon_i=0$. Also, since
$\alpha_{\varepsilon}>1/2$,
$\Delta_0\xi_{n,1}=O_P(\sigma_{n,1}^2/n)=O_P(\sigma_{n,2})=o_P(\sqrt{n})$
and via (\ref{eq:delta2}),
$\sum_{i=1}^n\Delta_i^2=O_P(\sigma_{n,2})=o_P(\sqrt{n})$. Finally,
the choice of $\delta_n$ yields
$\delta_n\sigma_{n,1}=o_P(\sqrt{n})$. Therefore, part (b) follows
from Theorem \ref{thm:main-weak}. $\Box$

\subsection{Residual empirical process with estimated parameters}\label{sec:estimated}
Let us focus on the parametric regression model of Section
\ref{sec:parametric}. From Corollary \ref{cor:par}
$$
\frac{1}{\sigma_{n,2}}\sup_{x\in\mathbb{R}}|\hat K_n(x)|\convdistr
\sup_{x\in\mathbb{R}}|f_{\varepsilon}^{(1)}(x)|
(Z_2-\frac{1}{2}Z_1^2),
$$
for $\alpha_{\varepsilon}<1/2$. The above result can be used, in
principle, to test whether the errors
$\varepsilon_1,\ldots,\varepsilon_n$ are consistent with a given
distribution $F_\varepsilon{}$. If however $F_{\varepsilon}$ belongs
to, say, a one-parameter family
$\{F_{\varepsilon}(\cdot,{\theta}),\theta\in \mathbb{R} \}$, then
one needs to know the value of the parameter $\theta$. A
straightforward procedure would be to estimate it and use the
statistic
$$
\sup_{x\in
\mathbb{R}}|\sum_{i=1}^n1_{\{\hat\varepsilon_i<x\}}-F_{\varepsilon}(x;\hat\theta_n)|,
$$
where $F_{\varepsilon}(x;\hat\theta_n)$ is the distribution function
$F_{\varepsilon}(x)=F_{\varepsilon}(x;\theta)$ in which the
parameter $\theta$ has been replaced with its estimator
$\hat\theta_n$.

Therefore, this section is devoted to study the limiting behaviour
of
$$
\hat L_n(x):=\sum_{i=1}^n \left(\mathbf{1}_{\{\hat\varepsilon_i\le
x\}}-F_{\varepsilon}(x;\hat\theta_n)\right).
$$
The results below may be seen as counterpart to the asymptotic
results for
$$
L_n(x):=\sum_{i=1}^n \left(\mathbf{1}_{\{\varepsilon_i\le
x\}}-F_{\varepsilon}(x;\hat\theta_n)\right),
$$
see \cite{Kulik2008b} for results and references therein for more discussion on this approach.\\

Many estimators $\hat\theta_n$ of $\theta$ can be obtained with help
of partial sums $\sum_{i=1}^nH(\hat\varepsilon_i)$, where $H$ is a
function that does not depend on $n$. Let us note that from Theorem
\ref{thm:main} we may have two scenarios for
$\alpha_{\varepsilon}<1/2$:
\begin{itemize}
\item[{\rm (A)}] $\sigma_{n,2}^{-1}\left(\sum_{i=1}^nH(\hat\varepsilon_i)-\Ex[H(\varepsilon_i)]\right)$ converges in distribution to a nondegenerate random variable;
\item[{\rm (B)}] $\sigma_{n,2}^{-1}\left(\sum_{i=1}^nH(\hat\varepsilon_i)-\Ex[H(\varepsilon_i)]\right)=o_P(1)$.
\end{itemize}
\begin{ex}{\rm
Consider $H(u)=u^2$ which yields the estimator of $\Var
(\varepsilon)$. We obtain for $\alpha_{\varepsilon}<1/2$:
$$
\sigma_{n,2}^{-1}\left(\sum_{i=1}^nH(\hat\varepsilon_i)-\Ex[H(\varepsilon_i)]\right)\convdistr
\int f_{\varepsilon}^{(1)}(v) dH(v)
\left(Z_1^2-\frac{1}{2}Z_2^2\right)=2\left(Z_1^2-\frac{1}{2}Z_2^2\right).
$$
Consider now $H(u)=u^3$. We have for $\alpha_{\varepsilon}<1/2$:
$$
\sigma_{n,2}^{-1}\left(\sum_{i=1}^nH(\hat\varepsilon_i)-\Ex[H(\varepsilon_i)]\right)\convdistr
6\int v f_{\varepsilon}(v) dv \left(Z_1^2-\frac{1}{2}Z_2^2\right).
$$
Consequently, if $f_{\varepsilon}$ is symmetric, then the right hand
side is simply 0 and thus we are in scenario (B). }
\end{ex}
In what follows, we will write $f_{\varepsilon}(\cdot;\theta)$ to
indicate the density with the true parameter $\theta$.
\begin{cor}\label{cor:empirical-estimated}
Assume that
$\hat\theta_n=\frac{1}{n}\sum_{i=1}^nH(\hat\varepsilon_i)$ and
$\theta=\Ex[H(\varepsilon)]$. Under the conditions of Corollary
\ref{cor:par}, we have
$$
\frac{1}{\sigma_{n,2}}\hat L_n(x)\convweak
\left(f_{\varepsilon}^{(1)}(x;\theta)+f_{\varepsilon}(x;\theta)\int
f_{\varepsilon}^{(1)}(u;\theta)dH(u)\right)\left(Z_2-\frac{1}{2}Z_1^2\right),
$$
and
$$
n^{-1/2}\hat L_n(x)\convweak W_1(x)+f_{\varepsilon}(x;\theta) \int
W_1(u)\; dH(u),
$$
respectively for $\alpha_{\varepsilon}<1/2$ and
$\alpha_{\varepsilon}>1/2$, provided that the integrals at the right
hand sides are finite.
\end{cor}
\begin{rem}{\rm
In case $\alpha_{\varepsilon}>1/2$, one needs very restrictive conditions on finiteness of
$\int
W_1(u)\; dH(u)$. In principle, it requires that $H$ has a finite support.
}
\end{rem}
\begin{rem}{\rm
We note that rates of convergence for $\hat L_n(\cdot)$, residual
empirical process with estimated parameters, are the same as for
$\hat K_n(\cdot)$, the ordinary residual empirical process. This is
different as compared to $K_n(\cdot)$ and its "estimated" version;
see \cite{Kulik2008b}. }
\end{rem}

\textit{Proof of Corollary \ref{cor:empirical-estimated}.} We
conduct the proof for $\alpha_{\varepsilon}<1/2$. For a function
$g(x;\theta)$ denote by $\nabla_{\theta}^{r}g(x;\theta)$ its $r$th
order derivative with respect to $\theta$, evaluated at
$\theta=\theta$. In particular, $\nabla=\nabla^1$. Then
\begin{eqnarray*}
\hat L_n(x) = \hat K_n(x) +
n(\theta-\hat\theta_n)\nabla_{\theta}F_{\varepsilon}(x;\theta)
+\frac{1}{2}n(\theta-\hat\theta_n)^2\nabla_{\theta}^2F_{\varepsilon}(x;\theta_n^*)
\end{eqnarray*}
with some $\hat\theta_n^*$ such that
$|\hat\theta_n^*-\hat\theta_n|\le |\theta-\hat\theta_n^*|$.
Therefore
\begin{eqnarray*}
\lefteqn{\hat L_n(x) =\hat K_n(x)+f_{\varepsilon}(x;\theta)\left(\sum_{i=1}^n (\Ex[H(\varepsilon)]-H(\hat\varepsilon_i))\right)+o_P(\sigma_{n,2})}\\
&=& \hat K_n(x)-f_{\varepsilon}(x;\theta) \left(\int H(u) d\hat K_n(u)\right)+o_P(\sigma_{n,2})\\
&=& \hat K_n(x)+f_{\varepsilon}(x;\theta) \left(\int \hat K_n(u)
dH(u)\right)+o_P(\sigma_{n,2})
\end{eqnarray*}
and the result follows from Corollary \ref{cor:par}.

\subsection{Nonparametric regression}\label{sec:nonparametric}
Now, we establish the result for nonparametric regression case. It
is assumed that $m(\cdot)$ is estimated by the usual Nadaraya-Watson
estimator, i.e.
\begin{equation}\label{eq:m-estimator}
\hat m(x)=\hat m_b(x)=\frac{1}{nb \hat
f_b(x)}\sum_{j=1}^nY_jK_b(x-X_j),
\end{equation}
with
\begin{equation}\label{eq:density-estimation}
\hat f_b(x)=\frac{1}{nb}\sum_{j=1}^nK_b(x-X_j),
\end{equation}
where $K_b(\cdot)=K(\cdot/b)$ and $K(\cdot)$ is a positive kernel,
which fulfills standard conditions: $\int K(u)\; du=1$, $\int
uK(u)\; du=0$ and $\int u^2K(u)\; du<\infty$.

Here we shall assume for simplicity that 
\begin{itemize}
\item[{\rm (P1)}]
Predictors are i.i.d. 
\end{itemize}
Results can be extended to LRD stationary predictors using estimates
from \cite{KulikWichelhaus2010}.
\begin{cor}\label{cor:nonpar}
Assume {\rm (P1)} and {\rm (E)}. Assume that $F_{\eta}(\cdot)$ is 5
times differentiable with bounded, continuous and integrable
derivatives. Also, suppose that the bandwidth fulfills
\begin{equation}\label{eq:b-cond-2}
b+(nb)^{-1}\to 0,
\end{equation}
\begin{equation}\label{eq:b-cond}
b\sigma_{n,1}^2/n\to \infty .
\end{equation}
\begin{itemize}
\item[{\rm (a)}]
If $\alpha_{\varepsilon}<1/2$, and
\begin{equation}\label{eq:b-cond-1}
b^2n/\sigma_{n,1}   + nb^4/\sigma_{n,2} +
b^2\sigma_{n,1}/\sigma_{n,2}+\frac{\sigma_{n,1}^2}{\sigma_{n,2}^2nb}\to
0 .
\end{equation}
then
$$
\frac{1}{\sigma_{n,2}}\hat K_n(x)\convweak f_{\varepsilon}^{(1)}(x)
(Z_2-\frac{1}{2}Z_1^2),
$$
where $Z_1,Z_2$ are defined in {\rm (\ref{eq:enr-distr})}.
\item[{\rm (b)}]
If $\alpha_{\varepsilon}>1/2$, and
\begin{equation}\label{eq:b-cond-1b}
b^2n/\sigma_{n,1}   + nb^4/\sqrt{n} +
b^2\sigma_{n,1}/\sqrt{n}+\sigma_{n,1}^2/(n^2b)\to 0 .
\end{equation}
then $ n^{-1/2}\hat K_n(x)\convweak W_1(x) \; $.
\end{itemize}
\end{cor}
\begin{rem}{\rm
The condition (\ref{eq:b-cond-2}) is standard in nonparametric
estimation. With the standard bandwidth choice $b=Cn^{-1/5}$ (see
e.g. \cite{KulikLorek2010}) condition (\ref{eq:b-cond}) is valid for
$\alpha_{\varepsilon}<4/5$. Likewise, one can easily verify that
(\ref{eq:b-cond-1}) holds for $\alpha_{\varepsilon}<4/5$ as well and
so for all $\alpha_{\varepsilon}<1/2$. Finally, (\ref{eq:b-cond-1b})
holds for $1/5<\alpha_{\varepsilon}<4/5$ and so for all
$1/2<\alpha_{\varepsilon}<4/5$. }
\end{rem}
\textit{Proof of Corollary \ref{cor:nonpar}.} In the nonparametric
regression model we have
\begin{equation}\label{eq:Delta}
\Delta_i=\hat m_b(X_i)-m(X_i)=R_b(X_i)+\frac{1}{nb\hat
f_b(X_i)}\sum_{j=1}^nK_b(X_i-X_j)\varepsilon_j,
\end{equation}
where
\begin{equation}\label{eq:Rn}
R_b(y)=\frac{1}{nb\hat
f_b(y)}\sum_{j=1}^n(m(X_j)-m(y))K_b\left(y-X_j\right).
\end{equation}
Denote $\rho(y)=(mf)''(y)-m(y)f''(y)$. Uniformly over
$\{y:f(y)>0\}$,
\begin{equation}\label{eq:density-esp}
R_b(y)-\frac{b^2\kappa_2}{2}\frac{\rho(y)}{f(y)}=O(b^4(1+o_P(1))).
\end{equation}
Now, in the second part of (\ref{eq:Delta}), we may replace $\hat
f_b(X_i)$ with $f(X_i)$. This is allowed since, first, $\hat
f_b(\cdot)$ is the consistent estimator of $f(\cdot)$; second, since
$K(\cdot)$ has bounded support ${\cal I}$ and $f(x)>0$, $x\in {\cal
I}$. Define for $j\not=i$,
$$
L_{b}(X_i,X_j)=\frac{1}{bf(X_i)}K_b(X_i-X_j).
$$
We may write (recall that $L_{b}^*(X_i,X_j)$ is the centered version
of $L_{b}(X_i,X_j)$)
$$
\Delta_i=R_b(X_i)+\Exp[L_{b}(X_1,X_2)]\bar\varepsilon+\frac{1}{n}\sum_{j=1}^nL_{b}^*(X_i,X_j)\varepsilon_j.
$$
Using (\ref{eq:density-esp}) and (\ref{eq:varDelta}) below we argue
that
\begin{equation}\label{eq:delta-exp}
\Delta_i=O_P(b^2)+\bar\varepsilon+o_P(\sigma_{n,1}/n),
\end{equation}
uniformly in $i$, provided that (\ref{eq:b-cond}) holds. Therefore,
the conditions of Theorem \ref{thm:main} are fulfilled with
$\Delta_0=\bar\varepsilon$, $\delta_n= \sigma_{n,1}/n$ and
$V=Z_1$, as long as (\ref{eq:b-cond}) and the first part of (\ref{eq:b-cond-1}) hold.  \\

From (\ref{eq:delta-exp}),
$$ \sum_{i=1}^n\Delta_i^2=
n\bar\varepsilon^2+O_P(b^2\sigma_{n,1})+O_P(nb^4)+o_P(\sigma_{n,1}^2/n)=n\bar\varepsilon^2+o_P(\sigma_{n,2}),
$$
if $\alpha_{\varepsilon}<1/2$ and (\ref{eq:b-cond-1}) holds.
Likewise, if (\ref{eq:b-cond-1b}) holds and
$\alpha_{\varepsilon}>1/2$,
$$\sum_{i=1}^n\Delta_i^2=n\bar\varepsilon^2+o_P(\sqrt{n})=O_P(\sigma_{n,1}^2/n)+o_P(\sqrt{n})=o_P(\sqrt{n}).$$
Also, from Section \ref{sec:proof-of-sum-delta} we obtain
\begin{equation}\label{eq:sum-Delta}
\sum_{i=1}^n\Delta_i=\sum_{i=1}^n\varepsilon_i+o_P(\sigma_{n,2}\vee\sqrt{n}).
\end{equation}
This finishes the proof. \hfill $\Box$

\subsection{Conjecture on error density estimation}\label{sec:density}
We consider again the parametric regression model $Y_i=\beta_0+\beta _! X_1+\varepsilon_i$. Our goal is to
estimate the error density $f_{\varepsilon}$.
We use the standard Parzen-Rosenblatt estimator
\begin{equation}\label{eq:density}
\hat
f_{h,\Delta}(x)=\frac{1}{nh}\sum_{i=1}^nK_h(x-\hat\varepsilon_i).
\end{equation}
\begin{conj}
Assume {\rm (P)} and {\rm (E)} and that (\ref{eq:hatbeta1}) holds.
Furthermore, assume that $F_{\eta}(\cdot)$ is 5 times differentiable with
bounded, continuous and integrable derivatives. If $\alpha_{\varepsilon}<1/2$ and
\begin{equation}\label{eq:bias}
nh^5\to 0, \qquad \sigma_{n,2}h\to \infty.
\end{equation}
then
$$
\frac{n}{\sigma_{n,2}}\left(\hat
f_{h,\Delta}(x)-f(x)\right) \convdistr f_{\varepsilon}^{(2)}(x) (Z_2-\frac{1}{2}Z_1^2).
$$
\end{conj}
\begin{rem}{\rm 
The first part of (\ref{eq:bias}) is the standard condition which assures that a bias is negligible. As for the second part, note that
$$
\frac{n}{\sigma_{n,1}}\left(\hat
f_{h}(x)-f(x)\right) \convdistr f_{\varepsilon}^{(1)}(x) Z_1,
$$
where $\hat f_h$ is the Parzen-Rosenblatt estimator of $f_{\varepsilon}$ based on $\varepsilon_1,\ldots,\varepsilon_n$. The above result is valid if $\sigma_{n,1}h\to \infty$. In other words, $\sigma_{n,2}h\to\infty$ is a large bandwidth condition which assures that the estimator has LRD-type behaviour. Otherwise, if $\sigma_{n,2}h\to 0$ is should be expected that the rate of convergence is $\sqrt{nh}$. However, the methods of this paper are not applicable to such situation. The same applies to the case $\alpha_{\varepsilon}>1/2$. 

The reader is referred to \cite{Wu2002} and \cite[Section 3.2]{KulikWichelhaus2010} for precise results on kernel density estimation under long memory.
}
\end{rem}
{\it "Proof".}
Clearly $$n\left(\hat
f_{h,\Delta}(x)-{\rm E}[f_{\varepsilon}(x)]\right)=\frac{1}{h}\int K\left(\frac{x-v}{h}\right) d\hat K_n(x-vh). $$
Using (\ref{eq:conclusion-1a}), and integrating by parts we write the left hand side as
$$
\left(\varepsilon_{n,2}-\frac{1}{2}n\bar\varepsilon^2\right)\int K(v)f_{\varepsilon}^{(2)}(x-vh)dv +\frac{1}{nh}\int\left(\hat K_n(x-vh)-S_n(x-vh)\right)dK(v).
$$
Therefore, we expect 
$$
\frac{1}{\sigma_{n,2}}n\left(\hat
f_{h,\Delta}(x)-{\rm E}[f_{\varepsilon}(x)]\right) \convdistr f_{\varepsilon}^{(1)}(x) (Z_2-\frac{1}{2}Z_1^2) .
$$
This, however, requires a more precise $o_P$ bound in (\ref{eq:conclusion-1a}). 
\hfill $\Box$

\section{Simulation studies}\label{sec:simulation}
We conducted simulations justifying our results on asymptotic
behaviour of supremum of the empirical process of residuals
$\hat{K}_n(\cdot)$. First, we simulated $n=100$ i.i.d. random
variables $\ve_i$, $i=1,\ldots,n$ from $N(0,1)$ distribution. Then,
supremum $\sup_{x\in\mathbb{R}} K_n(x)$ was calculated. This
procedure was repeated 100 times. Quartiles and standard deviation
of the empirical distribution of the supremum was calculated. Next,
for the  same errors, model $Y_i=1+4X_i+\ve_i$ was considered, and
residuals were calculated using estimators of $\beta_0$, $\beta_1$
given in (\ref{eq:beta-estimator}). Also, for the same errors, we
assumed that $\beta_0=1$ is known. The same procedure was repeated
with errors following LRD Gaussian process with
$\alpha_{\varepsilon}\in \{0.2,0.4,0.6,0.8\}$. The results are given
in Table \ref{tab:1}.
\begin{itemize}
\item Column 3: For the empirical process $K_n$ based on errors, the variability of the supremum increases with the dependence, which is in agreement with
the asymptotic theory for the LRD-based empirical processes.
\item Column 4: We consider the empirical process $L_n$, where $F_{\varepsilon}(\cdot)$ is replaced with $F_{\varepsilon}(\cdot;\hat\theta_n)$, $\hat\theta_n$ being sample standard deviation based
on errors $\varepsilon_1,\ldots,\varepsilon_n$. The results are
similar to column 3. In other words, estimation of variance does not
influence asymptotic behaviour of the empirical process. This agrees
with theoretical results; see \cite[Remark 1.6]{Kulik2008b}. This
happens since variance can be estimated with rate $\sigma_{n,2}\vee
\sqrt{n}$, whereas the rate of convergence for $K_n(\cdot)$ is
$\sigma_{n,1}$.
\item Column 5: We consider the residual-based empirical process $\hat K_n$ in the linear regression model. Both slope and intercept are estimated.
We note that the variability for $\alpha_{\varepsilon}=0.8$ or
$\alpha_{\varepsilon}=0.6$ is almost the same as for i.i.d. case. In
other words, LRD does not play any role, which is in agreement with
Corollary \ref{cor:par}.
\item Column 6: Results for the residual-based empirical process $\hat L_n$ with estimated variance are similar as for $\hat K_n$. Recall that Corollary
\ref{cor:empirical-estimated} indicates that rates of convergence
for $\hat L_n$ is the same as for $\hat K_n$.
\item Column 7: We consider $\hat K_n$, but the intercept is assumed to be known. Results are similar to Column 3. In other words,
in case of known intercept the asymptotic behaviour of $\hat K_n$ is
similar to $K_n$; see Remark \ref{rem:changed-rates}.
\end{itemize}

\begin{table}[h]
\begin{center}
\begin{tabular}{||c||c||c|c|c|c|c|c||}\hline\hline
\T \B &  &$K_n$  & $L_n$ &$\hat K_n$& $\hat L_n$ & $\hat K_n$;
$\beta_0=1$ & $\hat L_n$; $\beta_0=1$ \\ \hline

&$Q_1$&0.0416  &  0.0392  & 0.0467  & 0.0413  & 0.0419  & 0.0376 \\

i.i.d. &$Q_3$ & 0.0880 & 0.0859   & 0.0656 & 0.0592 & 0.0873 &
0.0789\\

&$s$&0.0314 & 0.0315 & 0.0169 & 0.0146 & 0.0312 &
0.0313\\\hline\hline

& $Q_1$& 0.0307 & 0.0274 & 0.0473  & 0.0448 & 0.0346 &
0.0278 \\
$\alpha_{\varepsilon}=0.8$& $Q_3$&0.0994 & 0.0940 & 0.0686 & 0.0637
& 0.0963 &
0.0908 \\
&$s$&0.0484&0.0494 & 0.0149 &0.013  & 0.0494 &0.0504\\\hline\hline

& $Q_1$ & 0.0303 & 0.0150 & 0.0488 & 0.0447 & 0.0274  &
0.0147 \\
$\alpha_{\varepsilon}=0.6$ & $Q_3$&0.1285 & 0.1237 & 0.0718 & 0.0646
& 0.1303 &
0.1192 \\
& $s$&0.0758 & 0.0786 & 0.0151 & 0.0139 & 0.0772 &
0.0797\\\hline\hline

& $Q_1$&0.0062 & 0.0038 & 0.0504 & 0.0471 & 0.0072 &
0.0041 \\
$\alpha_{\varepsilon}=0.4$& $Q_3$&0.1471 & 0.1353 & 0.0784 & 0.0662
& 0.1479 &
0.1353 \\
&$s$&  0.0858 &0.0852 &0.0194 &0.0147 &0.0850 &0.0845\\\hline\hline

&$Q_1$& 0.0015 & 0.0023 & 0.0535 & 0.0418 & 0.0021 &
0.0017 \\
$\alpha_{\varepsilon}=0.2$ & $Q_3$ &0.2870 & 0.2714 & 0.0826 &
0.0645 & 0.2978 &
0.2770 \\
&$s$& 0.1911 &0.1851 & 0.0218 &0.0178 & 0.1906 &
0.1850\\\hline\hline

\end{tabular}
\end{center}
\caption{ Simulated values of different dispersion measures.
}\label{tab:1}
\end{table}

\section{Technical details}
Let ${\cal H}_i=\sigma(\eta_i,\eta_{i-1},\ldots)$. Let ${\bf
u}=(u_1,\ldots,u_n)$ be a vector of scalars. Define
$$
Z_n(x;{\bf u})=\sum_{i=1}^n\left(\mathbf{1}_{\{\varepsilon_i\le
x+u_i\}}-F_{\varepsilon}(x+u_i)\right) -
\sum_{i=1}^n\left(\mathbf{1}_{\{\varepsilon_i\le
x\}}-F_{\varepsilon}(x)\right).
$$
The process $Z_n(x;{\bf u})$ is written as
\begin{eqnarray}\label{eq:decomp-Zn}
\lefteqn{\qquad\qquad Z_n(x;{\bf u}) =  \sum_{i=1}^n\left(\mathbf{1}_{\{x<\varepsilon_i\le x+u_i\}}-\Exp\left[ \mathbf{1}_{\{x<\varepsilon_i\le x+u_i\}}|{\cal H}_{i-1}\right]\right) }\\
&&+ \sum_{i=1}^n \left(\Exp\left[ \mathbf{1}_{\{x<\varepsilon_i\le
x+u_i\}}|{\cal H}_{i-1}\right] - \Exp\left[
\mathbf{1}_{\{x<\varepsilon_i\le x+u_i\}}\right] \right)=:M_n(x;{\bf
u})+N_n(x;{\bf u}).\nonumber
\end{eqnarray}
Recall now that ${\bf\Delta}=(\Delta_1,\ldots,\Delta_n)$. Recalling
(\ref{eq:basic-decomposition}), we decompose
\begin{eqnarray}\label{eq:decomposition}
\lefteqn{\qquad\qquad\hat K_n(x)-K_n(x) =}\\
&=&
M_n(x;{\bf\Delta})+N_n(x;{\bf\Delta})+f_{\varepsilon}(x)\sum_{i=1}^n\Delta_i+
\frac{1}{2}f_{\varepsilon}^{(1)}(x)\sum_{i=1}^n\Delta_i^2+O\left(\sum_{i=1}^n\Delta_i^3\right).\nonumber
\end{eqnarray}
First, in Corollary \ref{cor:LRD} we will establish an asymptotic
expansion for the LRD part $N_n(x;{\bf\Delta})$. This will be done
by considering a special structure of $N_n(x;{\bf u})$ (see Lemma
\ref{lem:LRD} and (\ref{eq:LRD-2}) below) and then "replacing" ${\bf
u}$ with ${\bf\Delta}$ under proper assumptions for the latter.

Furthermore, we have to bound $M_n(x;{\bf\Delta})$. This will be
done by obtaining an uniform bound on $M_n(x;{\bf u})$. In this way,
we may utilize the martingale structure of the latter. Clearly,
$M_n(x;{\bf\Delta})$ is not a martingale. The bounds are given in
Lemma \ref{lem:mtg} and Lemma \ref{lem:mtg-2}.

\subsection{LRD part}
Denote ${\bf u}_0=u_0 {\bf 1}$, where ${\bf 1}$ is the vector of
dimension $n$, consisting of '1'. Recall that
$\xi_i=\varepsilon_i-\eta_i$ and $\xi_{n,r}$ is defined in the
analogous way as $\varepsilon_{n,r}$.

In the first lemma we deal with $N_n(x;{\bf u}_0)$. The proof is
included in Section \ref{sec:proof-1}.
\begin{lem}\label{lem:LRD}
Assume that $F_{\eta}(\cdot)$ is 3 times differentiable with
bounded, continuous and integrable derivatives. Then with some
$0<\nu<1/2$ and $\delta_n\to 0$,
\begin{equation}\label{eq:LRD}
\sup_{|u_0|\le
\delta_n^{1-\nu}}\sup_{x\in\mathbb{R}}\left|N_n(x;{\bf
u}_0)+f_{\varepsilon}^{(1)}(x)u_0
\xi_{n,1}\right|=O_P\left(\delta_n^{1-\nu}\left(\sigma_{n,2} \vee
\sqrt{n}\right)+\delta_n^{2(1-\nu)}\sigma_{n,1}\right).
\end{equation}
\end{lem}
Note now that the part $N_n(x,{\bf u})$ in (\ref{eq:decomp-Zn}) can
be written as
$$
N_n(x;{\bf u})=\sum_{i=1}^n
\left(F_{\eta}(x+u_i-\xi_i)-F_{\eta}(x-\xi_i) -\Exp
F_{\eta}(x+u_i-\xi_i)+\Exp F_{\eta}(x-\xi_i)\right).
$$
Let us choose ${\bf u}={\bf u}_0+(u_{01},\ldots,u_{0n})$. If $\max_i
(|u_{0i}|)=o(\delta_n)$, then applying first order Taylor expansion,
and noting that $\xi_i$, $i\ge 1$, is LRD moving average with the
same properties as $\varepsilon_i$, $i\ge 1$,
$$
N_n(x;{\bf u})-N_n(x;{\bf u}_0)=o(\delta_n)\sum_{i=1}^n
\left(f_{\eta}(x+u_0-\xi_i)- {\rm
E}f_{\eta}(x+u_0-\xi_i)\right)=o_P(\delta_n\sigma_{n,1}),
$$
uniformly in $u, u_0$ and $x$, since $f_{\eta}^{(1)}$ is bounded and
integrable. Combining this with (\ref{eq:LRD}), we have (recall
$\nu<1/2$)
\begin{equation}\label{eq:LRD-2}
\sup_{{\bf u}}\sup_{x\in\mathbb{R}}\left|N_n(x;{\bf
u})+f_{\varepsilon}^{(1)}(x)u_0
\xi_{n,1}\right|=O_P(\delta_n^{1-\nu}\left(\sigma_{n,2}\vee
\sqrt{n}\right))+o_P\left(\delta_n\sigma_{n,1}\right),
\end{equation}
where $\sup_{{\bf u}}$ is taken over all ${\bf u}$ such that
$$
{\bf u}={\bf u}_0+(u_{01},\ldots,u_{0n}), \qquad \max_i
(|u_{0i}|)=o(\delta_n), \qquad |u_0|=O(\delta_n^{1-\nu}).
$$
In this way we end up with the following corollary.
\begin{cor}\label{cor:LRD}
Assume that $F_{\eta}(\cdot)$ is 3 times differentiable with
bounded, continuous and integrable derivatives. Assume that
${\bf\Delta}$ can be written as $\Delta_0{\bf
1}+(\Delta_{01},\ldots,\Delta_{0n})$, where
$$\Delta_0=o_P(\delta_n^{1-\nu}), \qquad \max_i \Delta_{0i}=o_P(\delta_n).$$
Then
\begin{equation}\label{eq:LRD-3}
\sup_{x\in\mathbb{R}}\left|N_n(x;{\bf
\Delta})+f_{\varepsilon}^{(1)}(x)\Delta_0
\xi_{n,1}\right|=O_P(\delta_n^{1-\nu}\left(\sigma_{n,2}\vee\sqrt{n}\right))+o_P\left(\delta_n\sigma_{n,1}\right).
\end{equation}

\end{cor}
Noting that for $\alpha_{\varepsilon}<1/2$ we have
$\xi_{n,1}-\varepsilon_{n,1}=o_P(\sigma_{n,2})$, we may replace
$\xi_{n,1}$ with $\varepsilon_{n,1}$ in the statement of Theorem
\ref{thm:main}.
\subsubsection{Proof of Lemma \ref{lem:LRD}}\label{sec:proof-1} Let $F_{n,\xi}(\cdot)$ be an empirical
distribution function, associated with $\xi_1,\ldots,\xi_n$ and let
$F_{\xi}(\cdot)$, $f_{\xi}(\cdot)$ be, respectively, distribution
and density function of any of $\xi_i$. Note that $\xi_i$ and
$\eta_i$ are independent for each fixed $i$, and
$f_{\xi}*f_{\eta}=f_{\varepsilon}$. Recall that $\xi_{n,r}$ is
defined in the analogous way as $\varepsilon_{n,r}$; see
(\ref{eq:enr}). From (\ref{eq:var-enr}) we obtain that
$\xi_{n,1}=O_P(\sigma_{n,1})$.

Furthermore, let
$$
\tilde S_{n,p}(x)=\sum_{i=1}^n\left(\mathbf{1}_{\{\xi_i\le
x\}}-F_{\xi}(x)\right)+\sum_{r=1}^p(-1)^{r-1}F_{\xi}^{(r)}(x)\xi_{n,r}.
$$
Note that $\tilde S_{n,p}$ is defined in the same way as $S_{n,p}$,
but we use $\xi_i$'s in the former instead of $\varepsilon_i$'s in
the latter. Nevertheless, we conclude from (\ref{eq:Wu}) and
(\ref{eq:Wu-weak}) that for $\alpha_{\varepsilon}<1/2$,
\begin{equation}\label{eq:Wu-tilde}
\sigma_{n,2}^{-1}\tilde S_{n,1}(x)\convweak f_{\xi}^{(1)}(x)Z_2,
\end{equation}
where $Z_2$ is the same random variable as in (\ref{eq:enr-distr}).
Otherwise, if $\alpha_{\varepsilon}>1/2$, then
\begin{equation}\label{eq:Wu-weak-tilde}
n^{-1/2}\tilde S_{n,1}(x)\convweak \Psi(x),
\end{equation}
where $\Psi$ is a Gaussian process and the convergence is in the
Skorokhod topology.

We compute
\begin{eqnarray*}
\lefteqn{N_n(x;{\bf u}_0)= n\int \left(F_{\eta}(x+u_0-v)-F_{\eta}(x-v)\right)d(F_{n,\xi}(v)-F_{\xi}(v))} \\
&=&n\int \left(F_{n,\xi}(v)-F_{\xi}(v)\right)\left(f_{\eta}(x+u_0-v)-f_{\eta}(x-v)\right)\; dv\\
&=& n\int \left(F_{n,\xi}(v)-F_{\xi}(v)+f_{\xi}(v)\xi_{n,1}/n\right)\left(f_{\eta}(x+u_0-v)-f_{\eta}(x-v)\right)\; dv\\
 &&- \left(f_{\varepsilon}(x+u_0)-f_{\varepsilon}(x) \right)\xi_{n,1}\\
&=& \int \tilde S_{n,1}(v)\left( f_{\eta}(x+u_0-v)
-f_{\eta}(x-v)\right)\; dv- f_{\varepsilon}^{(1)}(x)u_0
\xi_{n,1}+O(u_0^2)\xi_{n,1}\\
&=& \int \tilde S_{n,1}(v) f_{\eta}^{(1)}(x-v)u_0(v)\; dv-
f_{\varepsilon}^{(1)}(x)u_0 \xi_{n,1}+O(u_0^2)\xi_{n,1},
\end{eqnarray*}
where $u_0(v)$ lies between $x-v$ and $x+u_0-v$. From (\ref{eq:Wu})
and (\ref{eq:Wu-weak}) we conclude that $\sup_v|\tilde
S_{n,1}(v)|=O_P(\sigma_{n,2}\vee \sqrt{n})$. Therefore, with a
$1>\nu>0$,
$$
\sup_{|u_0|\le \delta_n^{1-\nu}}\sup_x\left|N_n(x;{\bf
u}_0)+f_{\varepsilon}^{(1)}(x)u_0
\xi_{n,1}\right|=O_P\left(\delta_n^{1-\nu}(\sigma_{n,2}\vee
\sqrt{n})+\delta_n^{2(1-\nu)}\sigma_{n,1}\right).
$$
\hfill $\Box$
\subsection{Martingale part}\label{sec:mtg-part}
The proofs for martingale part are standard, in particular, they are
similar as in \cite{ChanLing2008}. However, some details are
different, since the main theorems involve non-standard scalings
$n^{-1/2}$ and $\sigma_{n,2}^{-1}$, rather than $\sigma_{n,1}^{-1}$.
\begin{lem}\label{lem:mtg}
Assume that $\|f_{\eta}\|_{\infty}<\infty$.
\begin{itemize}
\item[{\rm (a)}]
Let $x_r=r  {1\over \sigma_{n,2}}$. If $\alpha_{\varepsilon}<1/2$
and {\rm (\ref{eq:cond-1a})} holds, then
$$
\sup_{\bf u}\max_{r\in\mathbb{Z} }|M_n(x_r;{\bf
u})|=o_P(\sigma_{n,2}).
$$
\item[{\rm (b)}]
Let $x_r=r  {\epsilon\over \sqrt{n}}$ with $\epsilon>0$. If
$\alpha_{\varepsilon}>1/2$ and {\rm (\ref{eq:cond-1b})} holds, then
$$
\sup_{\bf u}\max_{r\in\mathbb{Z} }|M_n(x_r;{\bf u})|=o_P(\sqrt{n}).
$$
\end{itemize}
In both cases $\sup_{{\bf u}}$ is taken over all ${\bf u}$ such that
\begin{equation}\label{eq:cond-2}
{\bf u}={\bf u}_0+(u_{01},\ldots,u_{0n}), \qquad \max_i
(|u_{0i}|)=o(\delta_n), \qquad |u_0|=O(\delta_n^{1-\nu}).
\end{equation}
\end{lem}
Let
\begin{equation}\label{eq:def-An}
A_n(x;y)= \sum_{i=1}^n\left(\mathbf{1}_{\{\ve_i\leq y\}}
-F_\ve(y)-(\mathbf{1}_{\{\ve_i\leq x\}}-F_\ve(x))\right).
\end{equation}
The next lemma establishes tightness-like property of the empirical
process based on $\varepsilon_i$, $i\ge 1$. Note, however, that it
cannot be concluded directly from the tightness of
$\sigma_{n,1}^{-1}K_n(\cdot)$, since the different scaling is
involved.
\begin{lem}\label{lem:tightness}
Assume that $\|f_{\eta}\|_{\infty}<\infty$.
\begin{itemize}
\item If $\alpha_{\varepsilon}<1/2$, then
$ \sup_{|y-x|\le\sigma_{n,2}^{-1}}|A_n(x;y)|=o_P(\sigma_{n,2}). $
\item
If $\alpha_{\varepsilon}>1/2$, then $ \sup_{|y-x|\le\epsilon
n^{-1/2}}|A_n(x;y)|=O_P(\epsilon n^{-1/2}). $
\end{itemize}
\end{lem}
Combining Lemmas \ref{lem:mtg} and \ref{lem:tightness} we obtain the
following uniform behaviour of the martingale part.
\begin{lem}\label{lem:mtg-2}
Under the conditions of Lemma \ref{lem:mtg} we have
$$
\sup_{\bf u}\sup_{x\in\mathbb{R}}|M_n(x;{\bf
u})|=o_P(\sigma_{n,2})+O_P(\epsilon\sqrt{n}).
$$
\end{lem}
As in case of Corollary \ref{cor:LRD} we conclude the following
corollary.
\begin{cor}\label{cor:mtg}
Assume that $\|f_{\eta}\|_{\infty}<\infty$. Assume that
${\bf\Delta}$ can be written as $\Delta_0{\bf
1}+(\Delta_{01},\ldots,\Delta_{0n})$, where
$$\Delta_0=o_P(\delta_n^{1-\nu}), \qquad \max_i \Delta_{0i}=o_P(\delta_n)$$
and that {\rm (\ref{eq:cond-1a})} or {\rm (\ref{eq:cond-1b})} holds
respectively for $\alpha_{\varepsilon}<1/2$ or
$\alpha_{\varepsilon}>1/2$. Then
\begin{equation}\label{eq:mtg}
\sup_{x\in\mathbb{R}}\left|M_n(x;{\bf
\Delta})\right|=o_P(\sigma_{n,2})+O_P(\epsilon\sqrt{n}).
\end{equation}
\end{cor}

\textit{Proof of Lemma \ref{lem:mtg}.} We prove part (a) only. The
proof of the other part is analogous. Let
$$a_{n,i}(x)=a_i(x):=\mathbf{1}_{\{x\leq \ve_t\leq x+u_i\}}
-\Ex [ \mathbf{1}_{\{x\leq \ve_i\leq x+u_i\}} |
\mathcal{H}_{i-1}]\;,$$ so that $M_n(x,\mathbf{u})= \sum_{i=1}^n
a_i(x)$. We note that $\{M_n(x,\mathbf{u}),{\cal H}_n\}$ is a
martingale array. Thus, by the Rosenthal's inequality
\begin{eqnarray*}
\Ex|M_n(x,\mathbf{u})|^4\leq
 C \Ex\left[ \left(\sum_{i=1}^n \Ex(a_i(x)^2|\mathcal{H}_{i-1})\right)^{2}\right]+C\sum_{i=1}^n \Ex
 a^4_i(x).
\end{eqnarray*}
Furthermore, $|a_i(x)|\le 1$, so that
\begin{eqnarray}\label{eq:bound-2}
\qquad\qquad\Ex|M_n(x,\mathbf{u})|^4\leq Cn\sum_{i=1}^n
\Ex\left[\left(\Ex(a_i^2(x)|\mathcal{H}_{i-1})\right)^2\right]
+C\sum_{i=1}^n \Ex a_i^2(x).
\end{eqnarray}
Note that
$$\Ex[a_i^2(x)|\mathcal{H}_{i-1}]\leq \Ex[\mathbf{1}_{\{\ve_i\leq x+|u_i|\}} | \mathcal{H}_{i-1}]
-\Ex[\mathbf{1}_{\{ \ve_i\leq x-|u_i|\}} |
\mathcal{H}_{t-1}]=:H^+_i(x)-H_i^-(x)$$ and that for each $i$,
$H_i^+(x)$ and $H_i^-(x)$ are nondecreasing.
\smallskip\par\noindent
Introduce a partition
$\mathbb{R}=\cup_{r\in\mathbb{Z}}[x_r,x_{r+1})$. Then
$$\Ex H^+_i(x_r)=\Ex H^+_i(x_r)\cdot \sigma_{n,2}\int_{x_r}^{x_{r+1}} 1\; dx\leq   \sigma_{n,2} \Ex\left[\int_{x_r}^{x_{r+1}} H^+_i(x)\; dx\right],$$
$$\Ex H^-_i(x_r)=\Ex H^-_i(x_r)\cdot \sigma_{n,2}\int_{x_r-1}^{x_{r}} 1\; dx\geq   \sigma_{n,2} \Ex\left[\int_{x_{r-1}}^{x_{r}} H^-_i(x) dx\right].$$
Thus, for arbitrary $M$,
\begin{eqnarray*}
\lefteqn{\sum_{r=-M}^M \Ex\left[H_i^+(x_r)-H_i^-(x_r)\right]\leq
\sigma_{n,2}\sum_{r=-M}^M \Ex\left[\int_{x_r}^{x_{r+1}} H^+_i(x)\;
dx-\int_{x_{r-1}}^{x_{r}} H^-_i(x) \; dx\right]}\\
 & =&\sigma_{n,2}
\Ex\left[\int_{x_{-M}}^{x_M} (H_i^+(x)-H_i^-(x))dx +
\int_{x_{M}}^{x_{M+1}} H_i^+(x) dx -\int_{x_{-M-1}}^{x_{-M}}
H_i^-(x)dx  \right]\\
&\le &  \sigma_{n,2} \Ex\left[\int_{x_{-M}}^{x_M}
(H_i^+(x)-H_i^-(x))dx\right]+2.
\end{eqnarray*}
Note that (recall that $\xi_i=\ve_i-\eta_i$)
\begin{equation}\label{eq:H_diff_int}
H^+_i(x)-H^-_i(x)=F_\eta(x-\xi_i+|u_i|)-F_\eta(x-\xi_i-|u_i|)=\int_{-|u_i|}^{|u_i|}
f_\eta(x-\xi_{i}+y)\; dy ,
\end{equation}
and
\begin{equation}\label{eq:H_diff_int_1}
|H^+_i(x)-H^-_i(x)|\le 2|u_i|\sup_x f_{\eta}(x).
\end{equation}
Using (\ref{eq:H_diff_int}) we obtain
\begin{eqnarray}\label{eq:bound-1}
\lefteqn{\sum_{r=-M}^M \Ex\left[H_i^+(x_r)-H_i^-(x_r)\right] }\\
&\le & 1 + \sigma_{n,2}
\Ex\left[\int_{x_{-M}}^{x_M}\int_{-|u_i|}^{|u_i|}f_\eta(x-\xi_{i}+y)\;
dy
\;dx\right]\nonumber\\
&\leq& 1 + \sigma_{n,2} \Ex\left[\int_{-|u_i|}^{|u_i|}
\int_{-\infty}^{\infty} f_\eta(x+\xi_{i}+y)\; dx\; dy\right]  \nonumber\\
&=&2 + \sigma_{n,2}
 \Ex\left[\int_{-|u_i|}^{|u_i|}  1\;
dy\right]=2+2 \sigma_{n,2}|u_i|.\nonumber
\end{eqnarray}
Combining (\ref{eq:H_diff_int}), (\ref{eq:H_diff_int_1}) and
(\ref{eq:bound-1}),
\begin{equation}\label{eq:bound-3}
\sum_{r=-M}^M \Ex\left[\left(
H_i^+(x_r)-H_i^-(x_r)\right)^2\right]\leq C|u_i|+C \sigma_{n,2}
u^2_i.
\end{equation}
Also, $\Ex a_i^2(x)\le \Ex[H_i^+(x_r)-H_i^-(x_r)]$. By Markov
inequality and (\ref{eq:bound-2}),
\begin{eqnarray*}
\lefteqn{P\left(\max_r {1\over \sigma_{n,2}}
|M_n(x_r,\mathbf{u})|>1\right)\leq {1\over \sigma_{n,2}^4 }\sum_r
\Ex M_n^4(x_r,\mathbf{u})={1\over \sigma_{n,2}^4 }\sum_r
\Ex\left(\sum_{i=1}^n
a_i(x_r)\right)^4 }\\
&\leq& {1\over \sigma_{n,2}^4 }\left\{ Cn\sum_{r}\sum_{i=1}^n
\Ex\left[\left(\Ex(a_i^2(x_r)|\mathcal{H}_{i-1})\right)^2\right]
+C\sum_r\sum_{i=1}^n \Ex a_i^2(x_r)\right\}\\
&\le & {C\over \sigma_{n,2}^4 }\left\{
n\sum_{i=1}^n|u_i|+n\sigma_{n,2}\sum_{i=1}^nu_i^2+n+\sigma_{n,2}\sum_{i=1}^n|u_i|\right\}.\qquad\qquad\qquad\qquad\qquad
\end{eqnarray*}
The bound converges to 0 under the conditions (\ref{eq:cond-1a}) and
(\ref{eq:cond-2}). \hfill $\Box$
\begin{proof}[Proof of Lemma \ref{lem:tightness}]
Similarly to (\ref{eq:decomp-Zn}), $A_n(x;y)$ is decomposed as
$\tilde M_n(x;y)+\tilde N_n(x;y)$, where $\tilde M_n(x;y)$ is the
martingale part and $\tilde N_n(x;y)$ is the LRD part. We have
\begin{equation}\label{eq:3}
\tilde N_n(x;y) =
\sum_{i=1}^n\left(\Ex[1_{\{x<\varepsilon_i<y\}}|{\cal
H}_{i-1}]-(F_{\varepsilon}(y)-F_{\varepsilon}(x))\right)\le n
\|f_{\eta}+f_{\varepsilon}\|_{\infty} |y-x|.
\end{equation}
From \cite[Lemma 14]{Wu2003}, $\sup_{|y-x|\le\epsilon
n^{-1/2}}|\tilde M_n(x;y)|=O_P(\epsilon n^{-1/2})$. Therefore, the
case $\alpha_{\varepsilon}>1/2$ is proven.

Furthermore, for $\alpha_{\varepsilon}<1/2$,
$$
\sup_{|y-x|\le\sigma_{n,2}^{-1}}|\tilde M_n(x;y)|\le 2\sup_{x\in
\mathbb{R}}\left|\sum_{i=1}^n\left(1_{\{\varepsilon_i\le
x\}}-\Ex\left[1_{\{\varepsilon_i\le x\}}|{\cal
H}_{i-1}\right]\right)\right|=o_P(\sigma_{n,2}).
$$
\end{proof}

\begin{proof}[Proof of Lemma \ref{lem:mtg-2}]
We start with $\alpha_{\varepsilon}<1/2$. We can rewrite $a_i(x)$ as
follows:
$$
a_i(x)=\mathbf{1}_{\{\ve_i\leq x+u_t\}}-\mathbf{1}_{\{\ve_i\leq
x\}}-(F_\eta(x-\xi_{i}+u_i)-F_\eta(x-\xi_{i})).
$$
Let $x\in[x_r,x_{r+1})$, since $\mathbf{1}_{\{\ve_i\leq x\}}$ and
$F_\eta(x)$ are nondecreasing functions with respect to $x$ we have
\begin{eqnarray*}
\lefteqn{ a_i(x)\leq \mathbf{1}_{\{\ve_i\leq
x_{r+1}+u_i\}}-\mathbf{1}_{\{\ve_i\leq x\}}-(
F_\eta(x-\xi_{i}+u_i)-F_\eta(x_{r+1}-\xi_{i}))}\\
&=& a_i(x_{r+1})+\mathbf{1}_{\{\ve_i\leq x_{r+1}\}}
-\mathbf{1}_{\{\ve_i\leq x\}} +F_\eta(x_{r+1}-\xi_{i}+u_i) -
F_\eta(x-\xi_{i}+u_i) .
\end{eqnarray*}
Thus, recalling the definition of $A_n(x;y)$ given in
(\ref{eq:def-An}),
\begin{eqnarray*}
  \lefteqn{M_n(x,\mathbf{u})=
M_n(x_r;  \mathbf{u})+\sum_{i=1}^n\left(\mathbf{1}_{\{\ve_i\leq
x_{r+1}\}} -F_\ve(x_{r+1})-(\mathbf{1}_{\{\ve_i\leq
x\}}-F_\ve(x))\right) }\\
  & & +\sum_{i=1}^n\left(F_\eta(x_{r+1}-\xi_{i}+u_i) - F_\eta(x-\xi_{i}+u_i)\right)\qquad\qquad\qquad\qquad\\
  &=:&  M_n(x_r;\mathbf{u}) + A_n(x;x_{r+1}) + B_n(x;x_{r+1};\mathbf{u}).\\
\end{eqnarray*}


\noindent Now,
\begin{eqnarray}\label{ineq:1}
\lefteqn{ \sup_{\bf
u}\sup_{x\in\mathbb{R}}|M_n(x;\mathbf{u})|=\sup_{\bf
u}\max_{r\in\mathbb{Z} }
\sup_{x\in[x_r,x_{r+1})}|M_n(x;\mathbf{u})|\leq \sup_{\bf u} \max_r |M_n(x_r;\mathbf{u})|}\nonumber\\
&  &+\sup_{|x_1-x_2|\leq \sigma_{n,2}^{-1}}A_n(x;x_{r+1}) +\sup_{\bf
u}\max_r\max_{x\in[x_r,x_{r+1})}B_n(x;x_{r+1};\mathbf{u}).
\end{eqnarray}
On account on Lemma \ref{lem:mtg}, the first term in (\ref{ineq:1})
is $o_P(\sigma_{n,2})$. The same holds for the second part
 by Lemma \ref{lem:tightness}.
 For last term we consider Taylor expansion for $F_\eta$:
$$ F_\eta(x_{r+1}-\xi_{i}+u_i) = F_\eta(x-\xi_{i}+u_i)+ f_\eta(s)(x_{r+1}-x), $$
where $s\in[x-\xi_i+u_i,x_{r+1}-\xi_i+u_i)$. Thus, the bound on
$B_n(x;x_{r+1};\mathbf{u})$ is independent of $\mathbf{u}$
$$B_n(x;x_{r+1};\mathbf{u})=\sum_{i=1}^n f_\eta(s)(x_{r+1}-x)\leq n f_\eta  (s){1\over \sigma_{n,2}} =o(\sigma_{n,2})$$
since $ n/\sigma^2_{n,2} \to 0$ for $\alpha_\ve<1/2$.
Thus, the proof for $\alpha_{\varepsilon}<1/2$ is finished.\\

If $\alpha_{\varepsilon}>1/2$, then with the choice
$x_r=r\frac{\epsilon}{\sqrt{n}}$ the first part in (\ref{ineq:1}) is
$o_P(\sqrt{n})$ and the same holds for the second part by applying
Lemma \ref{lem:tightness}. The term $B_n(x;x_{r+1};\mathbf{u})$ is
bounded by
$$B_n(x;x_{r+1};\mathbf{u})=\sum_{i=1}^n f_\eta(s)(x_{r+1}-x)\leq n f_\eta  (s){\epsilon\over \sqrt{n}} =O(\epsilon\sqrt{n}).$$
\end{proof}

\subsection{Proofs of Theorems \ref{thm:main} and \ref{thm:main-weak}}
The result of Theorem \ref{thm:main} follows from Corollary
\ref{cor:LRD} and uniform $o_P(\sigma_{n,2})$
negligibility of the martingale part in Lemma \ref{lem:mtg-2}.\\

Now, let $\alpha_{\varepsilon}>1/2$.
Corollary \ref{cor:mtg} implies that
for each $\eta,\theta>0$ we may choose $\epsilon>0$ small enough so
that
\begin{equation}\label{eq:2}
P\left(\sup_{x\in \mathbb{R}}|n^{-1/2}
M_n(x,{\bf\Delta})|>\theta\right)<1-\eta.
\end{equation}
Recall (\ref{eq:decomposition}). This combined with (\ref{eq:LRD-3})
of Corollary \ref{cor:LRD} and (\ref{eq:2}) yields
\begin{eqnarray}\label{eq:decomposition-1}
\lefteqn{\hat K_n(x)=K_n(x)+M_n(x;{\bf\Delta})
}\nonumber\\
&&+f_{\varepsilon}(x)\sum_{i=1}^n\Delta_i-f^{(1)}_{\varepsilon}(x)\Delta_{0}\xi_{n,1}
+ O\left(\sum_{i=1}^n\Delta_i^2\right)+O_P(\delta_n^{1-\nu}
\sqrt{n})+O_P\left(\delta_n\sigma_{n,1}\right),\nonumber\\
&=& K_n(x)+f_{\varepsilon}(x)\sum_{i=1}^n\varepsilon_i +\left(f_{\varepsilon}(x)\sum_{i=1}^n\Delta_i-f_{\varepsilon}(x)\sum_{i=1}^n\varepsilon_i-f^{(1)}_{\varepsilon}(x)\Delta_{0}\xi_{n,1}\right)   \nonumber\\
&&\qquad\qquad +O\left(\sum_{i=1}^n\Delta_i^2\right)+o_P(
\sqrt{n})+O_P\left(\delta_n\sigma_{n,1}\right).
\end{eqnarray}
Application of (\ref{eq:Wu-weak}) yields
$$
n^{-1/2}\left(K_n(x)+f_{\varepsilon}(x)\sum_{i=1}^n\varepsilon_i\right)\convweak
W_1(x).
$$
The result of Theorem \ref{thm:main-weak} follows.
\subsection{Proof of {\rm (\ref{eq:delta-exp})}}
We have
\begin{equation}\label{eq:Lb}
\Exp[L_{b}(X_1,X_2)]\sim 1+\frac{O(b^2)}{2}\int s^2K(s)\; ds
 \int f^{(2)}(v)\; dv =1+O(b^2).
\end{equation}
Consequently,
$\Exp[L_{b}(X_1,X_2)]\bar\varepsilon=\bar\varepsilon+O_P(b^2\sigma_{n,1}/n)=\bar\varepsilon+o_P(\sigma_{n,1}/n)$.

Furthermore, since central moments are bounded by ordinary moments,
\begin{eqnarray*}
\lefteqn{\Var\left(\frac{1}{n}\sum_{j=1}^nL_{b}^*(X_i,X_j)\varepsilon_j\right)=
}\\
&=& \frac{O(1)}{n^2}\sum_{j=1}^n \Exp[L_b^2(X_i,X_j)]
+\frac{1}{n^2}\sum_{j,j'=1\atop j\not=j'}^n
\Exp[L_b^*(X_i,X_j)\;L_b^*(X_i,X_{j'})]\Exp[\varepsilon_j\varepsilon_{j'}]\\
\end{eqnarray*}
It is straightforward to verify that for different indices $i,j,j'$,
$$
\Exp[L_b(X_i,X_j)L_b(X_i,X_{j'})]=1+O(b).
$$
Combining this with (\ref{eq:Lb}) yields
$$
\Exp[L_b^*(X_i,X_j)\;L_b^*(X_i,X_{j'})]=o(b).
$$
Consequently, if (\ref{eq:b-cond}) holds, then uniformly in $i$,
\begin{equation}\label{eq:varDelta}
\Var\left(\frac{1}{n}\sum_{j=1}^nL_{b}^*(X_i,X_j)\varepsilon_j\right)=O((nb)^{-1})+o(b\sigma_{n,1}^2/n^2)=o(\sigma_{n,1}^2/n^2).
\end{equation}

\subsection{Proof of
{\rm (\ref{eq:sum-Delta})}}\label{sec:proof-of-sum-delta} Recall
(\ref{eq:Delta}) and (\ref{eq:density-esp}). Also, recall that once
(\ref{eq:density-esp}) is evaluated, we may replace $\hat f_b(X_i)$
with $f(X_i)$. Therefore, we have
\begin{equation}\label{eq:sumDelta}
\sum_{i=1}^n\Delta_i=O_P(nb^2)+\Exp[L_b(X_1,X_2)]\varepsilon_{n,1}+\frac{1}{n}\sum_{j=1}^n
\tilde L_b^*(X_j)\varepsilon_j ,
\end{equation}
where $\tilde L_b(X_j)=\sum_{i=1}^n\frac{1}{bf(X_i)}K_b(X_i-X_j)$
and $\tilde L_b^*(X_j)$ is its centered version. Now, the variance
of the third term in (\ref{eq:sumDelta}) is
\begin{eqnarray*}
\lefteqn{\frac{1}{n^2}\sum_{j,j'=1}^n\Exp[\varepsilon_j\varepsilon_{j'}]\Exp\left[\tilde
L_{b}^*(X_j)\; \tilde L_{b}^*(X_{j'})\right]=I_1+I_2+I_3+I_4}\\
&:=& \frac{1}{n^2b^2}\sum_{j,j'=1\atop
j\not=j'}^n\Exp[\varepsilon_j\varepsilon_{j'}]
\sum_{i=1}^n\Cov\left[\frac{1}{f(X_i)}K_b(X_i-X_j),\frac{1}{f(X_i)}K_b(X_i-X_{j'})\right]\\
&&+\frac{1}{n^2b^2}\sum_{j,j'=1\atop
j\not=j'}^n\Exp[\varepsilon_j\varepsilon_{j'}] \sum_{i,i'=1\atop
i\not=i'}^n\Cov\left[\frac{1}{f(X_i)}K_b(X_i-X_j),\frac{1}{f(X_{i'})}K_b(X_{i'}\overline{}-X_{j'})\right]\\
&&+\frac{1}{n^2b^2}\sum_{j=1}^n\Exp[\varepsilon_j^2]
\sum_{i=1}^n\Cov\left[\frac{1}{f(X_i)}K_b(X_i-X_j),\frac{1}{f(X_i)}K_b(X_i-X_{j})\right]\\
&& +\frac{1}{n^2b^2}\sum_{j=1}^n\Exp[\varepsilon_j^2]
\sum_{i,i'=1\atop
i\not=i'}^n\Cov\left[\frac{1}{f(X_i)}K_b(X_i-X_j),\frac{1}{f(X_{i'})}K_b(X_{i'}\overline{}-X_{j})\right].
\end{eqnarray*}
We start with $I_1$. We claim that
$$
I_1=
\frac{O(1)}{n^2b^2}\sigma_{n,1}^2\left(\underbrace{nb^2}_{i,j,j' \;
\mbox{\rm\tiny different}}+  \underbrace{b}_{j\not=j', i=j \;
\mbox{\rm\tiny or } i=j'} \right)=O(\sigma_{n,1}^2/n).
$$
Indeed, let us verify the case when $i,j,j'$ are different. We have
(recall (P1))
\begin{eqnarray*}
\lefteqn{\Cov\left[\frac{1}{f(X_i)}K_b(X_i-X_j),\frac{1}{f(X_i)}K_b(X_i-X_{j'})\right]}\\
&\le &
\Exp\left[\frac{1}{f(X_i)}K_b(X_i-X_j)\frac{1}{f(X_i)}K_b(X_i-X_{j'})\right]\\
&=&\int\!\!\!\int\!\!\!\int
\frac{1}{f(u)}K_b(u-v)\frac{1}{f(u)}K_b(u-v')f(v)f(v')\; du\; dv\;d
v'\\
&=& b^2\int\!\!\!\int\!\!\!\int
\frac{1}{f(u)}K(s)\frac{1}{f(u)}K(s')f(u-sb)f(u-s'b)\; du\;
ds\;ds'=O(b^2).
\end{eqnarray*}
In $I_2$, the term with all indices $i,i',j,j'$ different vanishes
(recall that we work under (P1)). The other terms are verified in
the similar way as for $I_1$, by computing expected values of
products instead of covariances. We obtain:
$$
I_2=\frac{O(1)}{n^2b^2}\sigma_{n,1}^2
\left(\underbrace{nb}_{i\not=i',j\not=j',i=j,i'\not=j'}+
\underbrace{1}_{i\not=i',j\not=j',i=j,i'=j'}
\right)=O(\sigma^2_{n,1}/(nb)).
$$
Similarly,
$$
I_3=\frac{1}{n^2b^2}\sum_{j=1}^n\Exp[\varepsilon_j^2]\left(\sum_{i=1}^nO(b)+1\right)=O(b^{-1})=o(n).
$$
Finally, for $I_4$ let us note that with $i,i',j$ different we
obtain
$$\Cov\left[\frac{1}{f(X_i)}K_b(X_i-X_j),\frac{1}{f(X_{i'})}K_b(X_{i'}\overline{}-X_{j})\right]=0.$$
Thus,
$$
I_4=\frac{O(1)}{n^2b^2}n\left(  \underbrace{nb}_{i\not=i', i=j \;
\mbox{\rm\tiny or } i'=j} \right)=O(b^{-1})=o(n).
$$
From (\ref{eq:sumDelta}), (\ref{eq:Lb}) and the above estimates we
obtain
$$
\sum_{i=1}^n\Delta_i=O_P(nb^2)+\varepsilon_{n,1}+O_P(b^2\sigma_{n,1})+O_P\left(\frac{\sigma_{n,1}}{\sqrt{nb}}\right)+o_P(\sqrt{n})
.
$$
If $\alpha_{\varepsilon}<1/2$ and (\ref{eq:b-cond-1}) holds, then
the above estimate is $o_P(\sigma_{n,2})$. Likewise, if
$\alpha_{\varepsilon}>1/2$ and (\ref{eq:b-cond-1b}) holds, then the
bound is $o_P(\sqrt{n})$. Thus, (\ref{eq:sum-Delta}) is proven.
 \hfill $\Box$

\section*{Acknowledgement}
We would like to thank Professor Shiqing Ling for sending us the
correction note. The work of the first author was supported by a
NSERC (Natural Sciences and Engineering Research Council of Canada)
grant. The work of the second author was conducted while being a
Postdoctoral Fellow at the University of Ottawa.


\begin{thebibliography}{99}
\bibitem{Bai1994}
Bai, J. (1994). Weak convergence of the sequential empirical
processes of
              residuals in {ARMA} models. {\it Ann. Statist.} {\bf 22}, 2051--2061.



\bibitem{BeranGhosh1991}
Beran, J. and Ghosh, S. (1991). Slowly Decaying Correlations,
Testing Normality, Nuisance Parameters. {\it J. Amer. Statist.
Assoc.} {\bf 86}, 785--791.




\bibitem{Boldin1982}
Boldin, M. V. (1982). An estimate of the distribution of the noise
in an
              autoregressive scheme. {\it Teor. Veroyatnost. i Primenen.} {\bf 27}, 805--810.


\bibitem{ChanLing2008}
Chan, N. H. and Ling, S. (2008). Residual empirical processes for
long and short memory time
              series. \textit{Ann. Statist.} \textbf{36} (5),
              2453--2470. Correction note available at \begin{verbatim}
              http://www.imstat.org/aos/future_papers.html\end{verbatim}

\bibitem{Cheng2005}
Cheng, F. (2005). Asymptotic distributions of error density and
distribution
              function estimators in nonparametric regression. {\it J. Statist. Plann. Inference} {\bf 128}, 327--349.

\bibitem{DehlingTaqqu}
Dehling, Herold and Taqqu, Murad S. (1989).  The empirical process of some long-range dependent sequences with an
 application to $U$-statistics.
 \textit{Ann. Statist.} \textbf{17} (4), 1767--1783.



\bibitem{GuoKoul2008}
Guo, H. and Koul, H. L. (2008). Asymptotic inference in some
heteroscedastic regression models with long memory design and
errors. \textit{Ann. Statist.} \textbf{36} (1), 458--487.

\bibitem{HoHsing1996}
Ho, H.-C. and Hsing, T. (1996). On the asymptotic expansion of the
empirical process of long-memory moving averages. {\it Ann.
Statist.} {\bf 24}, 992--1024.

\bibitem{Koul}
Koul, K.L. (2002). \textit{Weighted empirical processes in dynamic nonlinear models.
Second edition of Weighted empiricals and linear models}.
Lecture Notes in Statistics, 166. Springer-Verlag, New York. 
		


\bibitem{KoulOssiander1994}
Koul, H. L. and Ossiander, M. (1994). Weak convergence of randomly
weighted dependent residual
              empiricals with applications to autoregression. {\it Ann. Statist.} {\bf 22}, 540--562.

\bibitem{KoulSurgailis1997}
Koul, H. L. and Surgailis, D. (1997). Asymptotic expansion of
{$M$}-estimators with long-memory
              errors. \textit{Ann. Statist.} {\bf 25}, 818--850.

\bibitem{KoulSurgailis2010}
Koul, H. L. and Surgailis, D. (2010). Goodness-of-fit testing under long-memory
              errors. \textit{J. Statist. Plann.
Inf.} \textbf{140}, 3742--3753.

\bibitem{Kulik2008b}
Kulik, R. (2008). Empirical process of long-range dependent
sequences when parameters are estimated. \textit{J. Statist. Plann.
Inf.} \textbf{139} (2), 287--294.

\bibitem{KulikLorek2010}
Kulik, R. and Lorek, P. (2011). Some results on random design
regression with long memory errors and predictors. \textit{J.
Statist. Plann. Inf.} \textbf{141}, 508-523.


\bibitem{KulikWichelhaus2010}
Kulik, R. and Wichelhaus, C. (2011). Nonparametric conditional variance and error density estimation in regression models with dependent errors and predictors. Revised for \textit{Electronic Journal of Statistics}.

\bibitem{HidalgoRobinson1997}
Robinson, P. M. and Hidalgo, F. J. (1997). Time series regression
with long-range dependence. \textit{Ann. Statist.} \textbf{25} (1),
77--104.


\bibitem{Schick2004}
M{\"u}ller, Ursula U. and Schick, Anton and Wefelmeyer, Wolfgang (2004). Estimating functionals of the error distribution in parametric and
 nonparametric regression. \textit{J. Nonparametr. Stat.}  \textbf{16} (3-4), 525--548.


\bibitem{Schick2009}
M{\"u}ller, Ursula U. and Schick, Anton and Wefelmeyer, Wolfgang (2009). 
Estimating the innovation distribution in nonparametric
              autoregression. \textit{Probab. Theory Related Fields}, \textbf{144} 
(1-2), 53--77. 


\bibitem{Wu2003}
Wu, W. B.  (2003). Empirical processes of long-memory sequences.
{\it Bernoulli} {\bf 9}, 809--831.

\bibitem{Wu2002}  
Wu, Wei Biao and  Mielniczuk, Jan (2002). Kernel density estimation for linear processes.
\textit{Ann. Statist.} \textbf{30} (5), 1441--1459.


\end{thebibliography}
\end{document}